\documentclass[a4paper,oneside,10pt,final,reqno]{amsart}

\usepackage{amssymb}
\usepackage[arrow,line,matrix]{xy}
\usepackage{version}
\usepackage{color}
\usepackage{showkeys}

\excludeversion{NB}

\newcommand{\bbP} {\mathbb{P}}
\newcommand{\bbQ} {\mathbb{Q}}
\newcommand{\bbZ} {\mathbb{Z}}

\newcommand{\calP} {\mathcal{P}}

\newcommand{\ve}{\varepsilon}
\newcommand{\wh}{\widehat}

\newcommand{\hgt}{\mathop{\operatorname{ht}}}

\newcommand{\ST}{\mathop{\operatorname{ST}}}

\newcommand{\Aut}{\mathop{\operatorname{Aut}} \nolimits}
\newcommand{\Cov}{\mathop{\operatorname{Cov}} \nolimits}

\newcommand{\vect}[1]{\underline{#1}}

\newcommand{\stir}[2]{\genfrac{\{}{\}}{0pt}{}{#1}{#2}}

\theoremstyle{plain}
 \newtheorem{thm}{Theorem}[section]
 \newtheorem{lem}[thm]{Lemma}
 \newtheorem{prop}[thm]{Proposition}
 
 \newtheorem{conj}[thm]{Conjecture}
 \newtheorem*{thm*}{Theorem}
\theoremstyle{definition}
 \newtheorem{dfn}[thm]{Definition}
 
 \newtheorem{fct}[thm]{Fact}
 \newtheorem{rmk}[thm]{Remark}
\theoremstyle{remark}

\numberwithin{equation}{section}
\allowdisplaybreaks

\begin{document}


\title{Integrality of the simple Hurwitz numbers}

\author{Shintarou Yanagida}
\address{Research Institute for Mathematical Sciences,
Kyoto University, Kyoto 606-8502, Japan}
\email{yanagida@kurims.kyoto-u.ac.jp}

\date{December 17, 2014}


\begin{abstract}
We show the integrality of the simple Hurwitz numbers.
The main tool is the cut-and-join operator,
and our proof is a purely combinatorial one.
\end{abstract}

\maketitle
\tableofcontents

\section{Introduction}

The purpose of this note is to prove the integrality of the simple Hurwitz numbers.
The statement is as follows.

\begin{thm*}[{Theorem \ref{thm:main}}]
The Hurwitz number $h_{g,\mu}$ is a positive integer except for the cases 
$h_{g,(1)} = 0$ $(g\in \bbZ_{\ge1})$ 
and $h_{g,(2)}=h_{g,(1,1)}=1/2$ $(g \in \bbZ_{\ge0})$.
\end{thm*}

By the term ``Hurwitz number",  we mean 
the weighted counting of the connected coverings of the projective line $\bbP^1$ 
which have a ramification at one point with the profile data given by $\mu$, and 
which also have simple ramifications at $r$ points.
Here $\mu$ is a partition and $g$ is a non-negative integer (the genus of the covering).
The number $r$ of simple ramifications is given by
$r := 2g-2+\ell(\mu)+|\mu|$,
where $\ell(\mu)$  is the length and $|\mu|$ is the total sum  of $\mu$.
Sometimes this number $h_{g,\mu}$ is called the simple Hurwitz number,
which appears in the title of this note.
See \S\ref{sect:hurwitz} for the detailed definition.

The main tool of our proof is the cut-and-join equation,
which will be explained in \S\ref{sect:cj}.
The content of \S\ref{sect:hurwitz} and \S\ref{sect:cj} 
is more or less known. 
The proof will be given in \S\ref{sect:main}.

\subsection*{Acknowledgements}

The author is supported by the Grant-in-aid for 
Scientific Research (No.\ 25800014), JSPS.
This work is also partially supported by the 
JSPS for Advancing Strategic International Networks to 
Accelerate the Circulation of Talented Researchers
``Mathematical Science of Symmetry, Topology and Moduli, 
  Evolution of International Research Network based on OCAMI"''.

This note is written during the author's stay at UC Davis. 
The author would like to thank the institute for support and hospitality.
He would also like to thank Professor Motohico Mulase 
for showing his interest and providing fundamental knowledge on the Hurwitz number.

\subsection*{Notation}

We denote by $\bbZ_{+} := \{1,2,3,\ldots\}$ the set of positive integers.

We follow \cite{M:1995} for the notations of partitions and related notions.
The Greek letters $\lambda,\mu,\ldots$ will be used for denoting partitions 
unless otherwise stated.
The symbol like $\vect{m}$ denotes a multi-index unless otherwise stated.

Here we recall a few notations of partitions in \cite{M:1995} 
which will be repeatedly used in the main text.
A partition is a finite sequence of positive integers 
ordered by size.
We include the empty sequence $\emptyset$ as a partition.
For a partition $\lambda = (\lambda_1,\lambda_2,\ldots,\lambda_k)$,
its length and total sum are denoted by 
$$
 \ell(\lambda):=k,\quad|\lambda| := \sum_{i=1}^k \lambda_i.
$$
The set of all partitions will be denoted by $\calP$.
We also denote by $\lambda \vdash n$  a partition $\lambda$ with $|\lambda| = n$.

A partition is identified with the Young diagram.
For $\lambda=(\lambda_1,\lambda_2,\ldots,\lambda_k)$,
the boxes are places at $\{(i,j) \in \bbZ^2 \mid 1\le i \le k, \ 1 \le j \le \lambda_i\}$.
Denote by $\lambda'$ the partition associated to the transposed Young diagram of $\lambda$.
The arm and leg for a box $\square$ with respect to the Young diagram 
associated to a partition $\lambda$ is denoted by $a_\lambda(\square)$ and $l_\lambda(\square)$.
If $\square$ is placed at $(i,j)$, then the arm and leg can be expressed as
$$
 a_\lambda(\square) := \lambda_i - j,\quad
 l_\lambda(\square) := \lambda'_j - i.
$$
Here if $i>\ell(\lambda)$ we set $\lambda_i := 0$.
The hook length is denoted by $h_\lambda(\square)$, so that 
$$
 h_\lambda(\square) = a_\lambda(\square) + l_\lambda(\square) + 1.
$$
Finally, 
$$
 z_\lambda := \prod_{n = 1}^{\infty} n^{m_n(\lambda)} m_n(\lambda)!,\quad
 m_n(\lambda) := \#\{ i \in \bbZ_{+} \mid \lambda_i = n\}.
$$

We also follow \cite{M:1995} for the notations of symmetric functions.
$\Lambda$ denotes the space of symmetric functions over $\bbZ$.
$s_\lambda$ denotes the Schur symmetric function associated to a partition $\lambda$.
Recall that the set $\{s_\lambda \mid \lambda \in \calP\}$ is a basis of $\Lambda$. 
$p_n$ denotes the $n$-th power sum symmetric function,
and we set $p_\lambda := \prod_{i=1}^{\ell(\lambda)} p_{\lambda_i}$ for 
a partition $\lambda$.
Recall also that the set $\{p_\lambda \mid \lambda \in \calP\}$ is a basis of 
$\Lambda_{\bbQ} := \Lambda\otimes_{\bbZ} \bbQ$.
If we want to express the set $x=(x_1,x_2,\ldots)$ of variables in symmetric functions, 
we will use the notation  
$$
 p_n[x] := x_1^n + x_2^n + \cdots.
$$

\section{Hurwitz number}
\label{sect:hurwitz}

Denote by $\Cov_d\bigl(\mu^{(1)},\ldots,\mu^{(r)}\bigr)$
the weighted number of $d$-fold coverings of $\bbP^1$
ramified over $r$ fixed points of $\bbP^1$ with monodromies in the conjugacy classes
corresponding to the partitions $\mu^{(i)}$.
The weight of a covering is the reciprocal of the order of its group of automorphisms.
Burnside formula says
\begin{align}\label{eq:Burnside:1}
 \Cov_d\bigl(\mu^{(1)},\ldots,\mu^{(r)}\bigr)
 = \sum_{\lambda \vdash d} 
    \left(\dfrac{\dim \lambda}{d!}\right)^2 
    \prod_{i=1}^r f_\lambda(\mu^{(i)}),\quad
 f_\lambda(\mu) := \dfrac{ \# C_\mu \chi^{\lambda}_\mu}{\dim \lambda}.
\end{align}
Here $C_\mu$ denotes the set of the conjugacy class corresponding to the partition $\mu$,
$\dim \lambda$ and $\chi^\lambda$ are the dimension and the character 
of the irreducible representation of 
the $d$-th symmetric group $\mathfrak{S}_{d}$ associated to 
the partition $\lambda$,
and $\chi^\lambda_\mu$ the value of $\chi^\lambda$ at the conjugacy class $C_\mu$.

Using the dimension formula (the so-called hook formula) 
$$
 \dim \lambda = \dfrac{|\lambda|!}{h_\lambda},\quad
 h^\lambda := \prod_{\square \in \lambda} h_\lambda(\square)
$$
and the counting formula of $C_\mu$
$$
 \# C_\mu = \dfrac{|\mu|!}{z_\mu},
$$
we can rewrite the formula \eqref{eq:Burnside:1} into the form 
\begin{align}\label{eq:Burnside:2}
 \Cov_d\bigl(\mu^{(1)},\ldots,\mu^{(r)}\bigr)
 = \sum_{\lambda \vdash d} 
    \left(\dfrac{\dim \lambda}{d!}\right)^2 
    \prod_{i=1}^r \dfrac{ h^\lambda \chi^{\lambda}_{\mu^{(i)}}}{z_{\mu^{(i)}}}.
\end{align}

Let us consider
$$
 \Cov_{d,r}(\mu) := 
 \Cov_{d}\bigl(\underbrace{(2,1^{d-2}),\ldots,(2,1^{d-2})}_{r},\mu\bigr)
$$ 
the weighted counting of coverings with simple ramifications at $r$ points 
and an arbitrary ramification at one point.
In \cite{O:2000} Okounkov discovered 
that the generating function of 
$\Cov_{d}\bigl((2,1^{d-2}),\ldots,(2,1^{d-2}),\mu,\nu\bigr)$
coincides with the $\tau$-function of the Toda lattice hierarchy.
Here we quote his result in the simplified version $\nu = \emptyset$.

\begin{fct}[{\cite{O:2000}}]
Let $Q$ and $\beta$ be (commuting) indeterminates and set 
$$
 \tau[x] := \sum_{r=0}^{\infty}\dfrac{\beta^r}{r!}\sum_{d=0}^\infty Q^r
  \sum_{\mu \vdash d} \Cov_{d,r}(\mu) p_\mu[x].
$$
Define $D^{(2)}$ to be the linear operator acting on the space $\Lambda$ of symmetric function as 
$$
 D^{(2)} s_\lambda = s_\lambda \kappa_\lambda/2, \quad
  \kappa_\lambda :=  2 f_\lambda\bigl((2,1^{d-2})\bigr)
$$
for arbitrary $\lambda \in \calP$.
Then we have
\begin{align}\label{eq:tau}
 \tau[x] = e^{\beta D^{(2)}}e^{Q p_1}.
\end{align}
\end{fct}

\begin{proof}
Using the expansion formula \cite[Chap. I, (7.7)]{M:1995} 
\begin{align}\label{eq:s-p}
 s_\lambda = \sum_{\mu \vdash |\lambda|} \dfrac{\chi^{\lambda}_\mu}{z_\mu} p_\mu
\end{align}
and the Burnside formula  \eqref{eq:Burnside:1} or \eqref{eq:Burnside:2},
one can easily find
$$
 \tau[x] = 
 \sum_{\lambda \in \calP} Q^{|\lambda|} e^{\beta \kappa_\lambda/2} 
 \dfrac{\dim \lambda}{|\lambda|!} s_\lambda[x].
$$
Recalling 
$$
 \dfrac{\dim \lambda}{|\lambda|!} = \left. s_\lambda \right|_{p_1=1,p_2=p_3=\cdots=0}
$$
we have
$$
 \tau[x] =  
  \sum_{\lambda \in \calP} Q^{|\lambda|} e^{\beta \kappa_\lambda/2} 
  s_\lambda[x]  \left. s_\lambda \right|_{p_1=1,p_2=p_3=\cdots=0}.
$$
Using the operator $D^{(2)}$ it is rewritten as
$$
 \tau[x] =  
  e^{\beta D^{(2)}} 
  \sum_{\lambda \in \calP} Q^{|\lambda|} 
  s_\lambda[x]  \left. s_\lambda \right|_{p_1=1,p_2=p_3=\cdots=0}.
$$
Finally by the Cauchy formula
$$
 \exp\left(\sum_{n=1}^{\infty}\dfrac{1}{n} p_n[x] p_n[y] \right)
 = \sum_{\lambda \in \calP} s_\lambda[x] s_\lambda[y],
$$
we have the result \eqref{eq:tau}.
\end{proof}

As a corollary, we can compute $\Cov_{d,r}(\mu)$ by the formula
\begin{align}\label{eq:covdr:formula}
 \Cov_{d,r}(\mu) = \text{coefficient of $p_\mu$ in } \  
  \dfrac{1}{|\mu|!} \left(D^{(2)}\right)^r p_1^{|\mu|}. 
\end{align}

Now we also note that the operator $D^{(2)}$ 
is nothing but the cut-and-join operator.

\begin{fct}
$D^{(2)}$ is realized by the following differential operator:
\begin{align}\label{eq:D2}
 D^{(2)} = \dfrac{1}{2} \sum_{k,l = 1}^{\infty}
  \left((k+l)p_k p_l \dfrac{\partial}{\partial p_{k+l}}
       + k l p_{k+l}\dfrac{\partial}{\partial p_k} \dfrac{\partial}{\partial p_l} \right).
\end{align}
In other words,
the Schur symmetric functions $s_\lambda$ are eigenfunctions of 
the differential operator $D^{(2)}$ with 
eigenvalues $\kappa_\lambda/2$.
\end{fct}

This eigenvalue  can be calculated from 
the definition $\kappa_\lambda/2 =  f_\lambda\bigl((2,1^{|\lambda|-2}))$ 
and the formula \cite[Chap. I, \S7, Example 5]{M:1995}
\begin{align}\label{eq:chi}
 \chi^{\lambda}_{\mu} = \sum_S (-1)^{\hgt(S)}.
\end{align}
Here the index $S$ runs over the set of all sequences of partitions 
$S=(\lambda^{(0)},\lambda^{(1)},\ldots,\lambda^{(m)})$ 
such that $m=\ell(\mu)$,
$0=\lambda^{(0)} \subset \lambda^{(1)} \subset \cdots \subset  \lambda^{(m)}$ 
and that $\lambda^{(i)} - \lambda^{(i-1)}$ is a border strip, 
and $\hgt(\lambda)$ is the number of rows occupied by $\lambda$ \emph{minus one}.
The result is 
\begin{align}\label{eq:kappa}
 \kappa_\lambda/2 =  \sum_{i\in\bbZ_+}\lambda_i(\lambda_i-2i+1)/2,
\end{align}
which is an integer.

\begin{rmk}
It seems that there are several ways to prove this statement.
A direct check can be seen, for example, at \cite[Proposition 2.2]{Zhou:2003}.
Here we also cite \cite[Theorem 3.1]{Stanley:1989},
which showed that the cut-and-join operator with one parameter 
$$
 D^{(2)}(\alpha) := 
 \dfrac{1}{2}\Biggl\{
  \sum_{k,l = 1}^{\infty}
   \Bigl((k+l)p_k p_l \dfrac{\partial}{\partial p_{k+l}}
       + \alpha k l p_{k+l}\dfrac{\partial}{\partial p_k} \dfrac{\partial}{\partial p_l} \Bigr)
 + (\alpha-1) \sum_{k-1}^{\infty} k^2 p_k \dfrac{\partial}{\partial p_k}
 \Biggr\}
$$
has the Jack symmetric functions as eigenfunctions.
Let us denote by $P_\lambda(x;\alpha)$ the (monomial) Jack symmetric function 
normalized as $P_\lambda(x;\alpha) = m_\lambda + (\text{lower terms})$,
where $m_\mu$ is the monomial symmetric functions and 
the ordering is taken to be the dominance ordering
(see \cite[Chap. VI, \S10]{M:1995}).
Then we have 
$$
 D^{(2)} P_\lambda(x;\alpha) = P_\lambda(x;\alpha) \ve_\lambda(\alpha), \quad
 \ve_\lambda(\alpha) := \alpha n(\lambda') - n(\lambda) +(\alpha-1)|\lambda|/2,
$$
where $n(\lambda) := \sum_i (i-1) \lambda_i = \sum_i \binom{\lambda'_i}{2}$.
Since $P_\lambda(x;1)=s_\lambda$, we recover the statement \eqref{eq:D2}
\end{rmk}

Finally we introduce the Hurwitz numbers $h_{g,\mu}$ as follows.

\begin{dfn}
Define $h_{g,\mu} \in \bbQ$ for $g \in \bbZ_{\ge0}$ and $\mu \in \calP$ by
$$
 \log \tau[x] = \sum_{d,r=0}^{\infty} Q^d \dfrac{\beta^r}{r!}
  \sum_{\mu \vdash d} h_{g,\mu} p_\mu[x]
$$
with the relation
\begin{align}\label{eq:gr}
 2g - 2 + \ell(\mu) + |\mu| = r.
\end{align}
\end{dfn}

By the property of $\log$, 
one finds that $h_{g,\mu}$ counts the \emph{connected} coverings 
with simple ramifications at $r$ points and the ramification given by $\mu$ 
at one point.

By this definition and the formula \eqref{eq:covdr:formula}, 
one can compute the Hurwitz numbers.
We list here the numbers for $g \le 6$ and $|\mu| \le 6$ in 
Tables \ref{table:mu<=5} and \ref{table:mu=6}.
They (of course) coincide with the results in \cite{LZZ:2000} and \cite{EMS:2011}.

\begin{table}[htbp]
\begin{align*}
{\small 
\begin{array}{r||rrrrrrr}
&g=0 & 1 & 2 & 3 & 4 & 5 & 6 
\\
\hline
\hline
 \mu=(1)&1&0&0&0&0&0&0 
\\
\hline
 (2)&1/2&1/2&1/2&1/2&1/2&1/2&1/2
\\
(1^2)&1/2&1/2&1/2&1/2&1/2&1/2&1/2
\\
\hline
(3)&1&9&81&729&6561&59049&531441
\\
(2,1)&4&40&364&3280&29524&265720&2391484
\\
(1^3)&4&40&264&3280&29524&265720&2391484
\\
\hline
(4)&4&160&5824&209920&7558144&272097280&9795518464
\\
(3,1)&27&1215&45927&1673055&60407127&2176250895&78359381127
\\
(2,2)&12&480&17472&629760&22674432&816291840&29386555392
\\
(2,1^2)&120&5460&206640&7528620&271831560&9793126980&352617206880
\\
(1^4)&120&5460&206640&7528620&271831560&9793126980&352617206880
\\
\hline
(5)&25&3125&328125&33203125&3330078125&333251953125&33331298828125
\\
(4,1)&256&35840&3956736&409108480&41394569216&4156871147520&416314027933696
\\
(3,2)&216&26460&2748816&277118820&27762350616&2777408868780&277768823459616
\\
(3,1^2)&1620&234360&26184060&2719617120&275661886500&27700994510280&2774997187556940
\\
(2^2,1)&1440&188160&20160000&2059960320&207505858560&20803767828480&2082272553861120
\\
(2,1^3)&8400&1189440&131670000&13626893280&1379375197200&138543794363520&13876390744734000
\\
(1^5)&8400&1189440&131670000&13626893280&1379375197200&138543794363520&13876390744734000
\end{array}}
\end{align*}
\caption{The Hurwitz numbers $h_{g,\mu}$ for $g\le6$ and $|\mu|\le5$}
\label{table:mu<=5}.
\end{table}

\begin{table}[htbp]
\begin{align*}
{\tiny
\begin{array}{r||rrrrrrr}
&g=0 & 1 & 2 & 3 & 4 & 5 & 6 
\\
\hline
\hline
\mu=(6)&216&68040&16901136&3931876080&895132294056&202252053177720&45575342328002976
\\
(5,1)&3125&1093750&287109375&68750000000&15885009765625&3615783691406250&817717742919921875
\\
(4,2)&2560&788480&192783360&44490434560&10093234511360&2277308480778240&512887872299714560
\\
(4,1^2)&26880&9838080&2638056960&638265788160&148222087453440&33821881625226240
&7657985270680120320
\\
(3^2)&1215&357210&86113125&19797948720&4487187539835&1012204758777030&227953607360883345
\\
(3,2,1)&45360&14696640&3710765520&872470478880&199914163328880&45334411650702720
&10235275836481639440
\\
(3,1^3)&181440&65998800&17634743280&4259736280800&988561437383520&225514718440830000
&51056208831963782160
\\
(2^3)&6720&2016000&486541440&111644332800&25269270586560&5696315163302400&1282471780397902080
\\
(2^2,1^2)&241920&80438400&20589085440&4874762692800&1120875021826560&254613060830419200
&57531761566570529280
\\
(2,1^4)&1088640&382536000&100557737280&24109381296000&5576183206513920&1270116357617016000
&287353806073982746560
\\
(1^6)&1088640&382536000&100557737280&24109381296000&5576183206513920&1270116357617016000
&287353806073982746560
\end{array}}
\end{align*}
\caption{The Hurwitz numbers $h_{g,\mu}$ for $g\le6$ and $|\mu|=6$}
\label{table:mu=6}
\end{table}

One can observe various properties of the values of $h_{g,\mu}$ from 
these tables.
Here we list a few formula which are more or less known. 

\begin{lem}\label{lem:hgn:known}
\begin{enumerate}
\item 
For any $n\in\bbZ_+$ we have
$$h_{0,(n)} = n^{n-3}.$$

\item 
For any $g \in \bbZ_{\ge1}$ we have
$$
 h_{g,(1)} = 0.
$$

\item
For any $g \in \bbZ_{\ge0}$ we have
$$
 h_{g,(2)} = h_{g,(1,1)} = 1/2.
$$

\item
For any $g \in \bbZ_{\ge0}$ and $n \in \bbZ_{\ge2}$ 
we have
$$
 h_{g,(2,1^{n-2})} = h_{g,(1^n)}.
$$
\end{enumerate}
\end{lem}

We close this section by giving a proof of (1).
The others will be shown with the help of 
the cut-and-join equation in the next section.

\begin{proof}[{Proof of Lemma \ref{lem:hgn:known} (1)}]
By the definition we have 
$$
 \Cov_{d,r}\bigl((d)\bigr) = h_{g,(d)},
$$
where $r$ and $g$ are related by $r=2g-1+d$ as \eqref{eq:gr}.
Thus from \eqref{eq:covdr:formula} it is enough to compute
$$
 h_{0,(n)} = \text{coefficient of $p_n$ in } \ 
  \dfrac{1}{n!} \left(D^{(2)}\right)^{n-1} p_1^n.
$$

Recall the Pieri formula \cite[Chap. I, (5,17)]{M:1995} 
$s_\lambda e_r = \sum_\mu s_\mu$
for a  partition $\lambda$ and a  positive integer $r$.
Here the running index $\mu$ runs over the set of partitions of $|\lambda| +n$
such that $\mu \supset \lambda$ and the skew Young diagram $\mu -\lambda$ 
is a vertical strip.
Using the Pieri formula repeatedly  for $r=1$, one can see
$$
 p_1^n = s_{(1)}^n = \sum_{\lambda\vdash n} \#\ST(\lambda) s_\lambda,
$$
where $\ST(\lambda)$ denotes the set of standard tableaux with shape $\lambda$.
Thus we have 
\begin{align*}
 h_{0,(n)}&= \text{coefficient of $p_n$ in } \ 
  \dfrac{1}{n!} \left(D^{(2)}\right)^{n-1} \sum_{\lambda\vdash n} \#\ST(\lambda) s_\lambda,
\\
&= \text{coefficient of $p_n$ in } \ 
  \dfrac{1}{n!} \sum_{\lambda\vdash n} 
   \#\ST(\lambda) \left(\dfrac{\kappa_{\lambda}}{2}\right)^{n-1} s_\lambda.
\end{align*}

Next we recall the expansion \eqref{eq:s-p}
 $s_\lambda = \sum_{\mu \vdash |\lambda|} s_\mu \chi^\lambda_\mu/z_\mu$ 
and the formula \eqref{eq:chi} of $\chi^{\lambda}_{\mu}$.
In the case $\mu = (n)$ the latter gives 
$$
 \chi^\lambda_{(n)} = 
 \begin{cases}
  (-1)^{n-r} & \lambda = (r,1^{n-r}) \ \text{ for some } 1\le r \le n  \\
  0 & \text{otherwise}
 \end{cases}.
$$
Thus we have
\begin{align*}
h_{0,(n)}
&=\dfrac{1}{n!} \sum_{r=1}^{n} 
   \#\ST\bigl((r,1^{n-r})\bigr) \left(\dfrac{\kappa_{(r,1^{n-r})}}{2}\right)^{n-1} 
   \dfrac{(-1)^{n-r}}{n}.
\end{align*}
Since $\#\ST\bigl((r,1^{n-r})\bigr) = \binom{n-1}{r-1}$ by an easy observation,
we have
\begin{align}\label{eq:lem:h0n:int}
h_{0,(n)}
&=\dfrac{1}{n!\cdot n} \sum_{r=1}^{n} 
   (-1)^{n-r} \binom{n-1}{r-1} \left(\dfrac{n(2r-n-1)}{2}\right)^{n-1}
=\dfrac{1}{n!\cdot n} \sum_{s=0}^{n-1} 
   (-1)^{s} \binom{n-1}{s} \left(\binom{n}{2}- n s\right)^{n-1}.
\end{align}
Here we used the formula \eqref{eq:kappa} for $\kappa_\lambda$.

Finally, using the elementary formula
\begin{align}\label{eq:fmp}
f_{m,p} := \sum_{s=0}^m \binom{m}{s} (-1)^s s^p = 
\begin{cases}
0 & 1 \le p \le m-1 \\ 
(-1)^m m! & p = m
\end{cases},
\end{align}
we have 
\begin{align*}
h_{0,(n)}
=\dfrac{ (- n)^{n-1}}{n!\cdot n}  f_{n-1,n-1} = n^{n-3}.
\end{align*}
\end{proof}

By a slight generalization of the above argument,
we have 

\begin{prop}\label{prop:int:hg(n)}
For any $g\in\bbZ_{\ge0}$ and $n \in \bbZ_{\ge3}$,
$h_{g,(n)}$ is a positive integer.
\end{prop}

\begin{proof}
Denote by  $\stir{k}{j} := \sum_{i=0}^j (-1)^{j-i}\binom{j}{i}i^n/j!$ 
the Stirling number of the second kind.
It satisfies $x^k = \sum_{j=0}^k \stir{k}{j} x(x-1)\ldots (x-j)$.
Then the formula \eqref{eq:fmp} can be generalized to 
\begin{align}\label{eq:fmp:2}
 f_{m,p} = \stir{p}{m} (-1)^m m!
\end{align}
for arbitrary $m,p \in \bbZ_{\ge0}$.
Using this generalization, 
one can prove
\begin{align}\label{eq:hgn:l=1:int}
 h_{g,(n)} \in \bbZ \ \text{ for any }\  g\ge0, \ n\ge3.
\end{align}
In fact, 
by the same way to obtain \eqref{eq:lem:h0n:int},
one has
\begin{align*}
h_{g,(n)}
=\dfrac{1}{n!\cdot n} \sum_{s=0}^{n-1} 
   (-1)^{s} \binom{n-1}{s} \left(\binom{n}{2}- n s\right)^{n-1+g}.
\end{align*}
Then \eqref{eq:fmp:2} yields
\begin{align*}
h_{g,(n)}
&=\dfrac{1}{n!\cdot n} \sum_{p=n-1}^{n-1+g} 
  \binom{n-1+g}{p}\binom{n}{2}^{n-1+g-p}(-n)^p f_{n-1,p}\\
&=n^{n-3}\sum_{r=0}^{g} 
  \binom{n-1+g}{n-1+r}\binom{n}{2}^{g-r}(-n)^{r} \stir{n-1+r}{n-1}.
\end{align*}
The last summation consists of integers,
so we have the result \eqref{eq:hgn:l=1:int}.

Using the cut-and-join relations which will be introduced in the next section,
one finds that $h_{g,(n)}$ is positive,
so that in \eqref{eq:lem:h0n:int} we actually have 
$ h_{g,(n)} \in \bbZ_{+}$.
\end{proof}

\section{Cut-and-join equation}
\label{sect:cj}

We now introduce the cut-and-join equation following \cite[\S3]{EMS:2011}.
There are a large amount of literature for this topic,
and here we only cite \cite{GJ:1997}.

We extend the definition of the Hurwitz number $h_{g,\mu}$ 
to the case when $\mu$ is  a multi-index: For 
$\vect{k} = (k_1,k_2,\ldots,k_n) \in \bbZ_{+}^n$, 
we define 
$$
 h_{g,\vect{k}} := h_{g,\mu},
$$
where $\mu$ is the partition obtained from $\vect{k}$ 
by ordering numbers $k_i$'s in size.
Thus $h_{g,\vect{k}}$ is invariant 
under permutation of the parts of $\vect{k}$.

Hereafter we  use the same symbols $\ell(\vect{k})$ and $|\vect{k}|$ 
for the length and the total sum of $\vect{k}$ 
as in the case of partitions.
Set 
$$
 H_{g}(\vect{k}) := \dfrac{\# \Aut(\vect{k})}{r(g,\vect{k})!} h_{g,\vect{k}}.
$$
Here $\Aut(\vect{k})$ is the set of automorphism of $\vect{k}$,
so we have 
$$
 \# \Aut \vect{k} = \prod_{i \in \bbZ_+} m_i(\vect{k})!,\quad
 m_i(\vect{k}) := \#\{j \in \bbZ_+ \mid k_j = i\}.
$$
The function $r(g,\vect{k})$ is defined by
$$
 r(g,\vect{k}) := 2g-2+\ell(\vect{k})+|\vect{k}|.
$$

\begin{fct}[{\cite[Theorem 3.1]{MZ:2010}}]\label{fct:cj}
$H_{g}(\vect{k})$ satisfies a recusion equation
\begin{align*}
 r(g,\vect{k}) H_{g}(\vect{k})
=&\sum_{i < j} (k_i+k_j) H_{g}\bigl(\vect{k}(\wh{i},\wh{j};k_i+k_j)\bigr)
\\
 &+\dfrac{1}{2}\sum_{i=1}^{\ell} 
   \sum_{\alpha+\beta=k_i} \alpha \beta
  \left(H_{g-1}\bigl(\vect{k}(\wh{i};\alpha,\beta)\bigr)
   +\sum_{ g_1+g_2=g}
    \sum_{\vect{m} \sqcup \vect{n} = \vect{k}(\wh{i})}
    H_{g_1}\bigl(\vect{m}(\alpha)\bigr) H_{g_2}\bigl(\vect{n}(\beta)\bigr) 
   \right).
\end{align*}
Here we set $\ell := \ell(\vect{k})$ and 
\begin{align*}
\vect{k}(\wh{i},\wh{j};k_i+k_j) &:= 
 (k_1,\ldots,\widehat{k_i},\ldots,\widehat{k_j},\ldots,k_\ell,k_i+k_j),\\
\vect{k}(\wh{i};\alpha,\beta) &:= 
 (k_1,\ldots,\widehat{k_i},\ldots,k_\ell,\alpha,\beta),\\
\vect{k}(\wh{i}) &:= 
 (k_1,\ldots,\widehat{k_i},\ldots,k_\ell),\\
\vect{m}(\alpha) &:= 
 (m_1,\ldots,m_{\ell(\vect{m})},\alpha).
\end{align*}
In the summations, $\alpha,\beta$ are positive integers,
$g_1,g_2$ are non-negative integers 
and $\vect{m},\vect{n}$ are (possibly empty) 
ordered subsets of $\vect{k}(\wh{i})$, 
where  $\vect{k}(\wh{i})$ is considered as an ordered set of integers.
\end{fct}

\begin{rmk}\label{rmk:cj:rec}
\begin{enumerate}
\item
This equation follows from 
$\partial_\beta \tau[x] = D^{(2)} \tau[x]$,
which is obvious from $\tau[x]=e^{\beta D^{(2)}} e^{Q p_1[x]}$.

\item
Let us see the change of $r(g,\vect{k})$ by this equation.
Set $\ell:=\ell(\vect{k})$ and $n:=n(\vect{k})$. Then
\begin{align*}
&r(g,\vect{k}) = 2g-2+\ell+n, \\
&r\bigl(g,\vect{k}(\wh{i},\wh{j};k_i+k_j)\bigr) =2g-2+(\ell-1)+n = r(g,\vect{k})-1, \\
&r\bigl(g-1,\vect{k}(\wh{i};\alpha,\beta)\bigr) =2(g-1)-2+(\ell+1)+n = r(g,\vect{k})-1, \\
&r\bigl(g_1,\vect{m}(\alpha)\bigr)  \le 2g-2+\ell+(n-1) = r(g,\vect{k})-1.
\end{align*}
Thus the cut-and-join always decreases $r(g,\vect{k})$,
so that it gives a recursive computation for $h_{g,\vect{k}}$
with the initial value $h_{0,(1)}=1$.
\end{enumerate}
\end{rmk}

An easy application of the cut-and-join equation is 
the proofs of Lemma \ref{lem:hgn:known} (2), (3) and (4).
Let us recall
\begin{enumerate}
\setcounter{enumi}{1}
\item 
For any $g \in \bbZ_{\ge1}$ we have
$$
 h_{g,(1)} = 0.
$$
\item
For any $g \in \bbZ_{\ge0}$ we have
$$
 h_{g,(2)} = h_{g,(1,1)} = 1/2.
$$
\item
For any $g \in \bbZ_{\ge0}$ and $n \in \bbZ_{\ge2}$ we have
$$
 h_{g,(2,1^{n-2})} = h_{g,(1^n)}.
$$
\end{enumerate}

\begin{proof}[{Proof of Lemma \ref{lem:hgn:known} (2)}]
The cut-and-join equation for $\ell(\vect{k})=1$ is simply
\begin{align}\label{eq:cj:l=1}
 (2g-1+n)H_g\bigl((n)\bigr)
=\dfrac{1}{2}\sum_{\alpha+\beta=n}\alpha\beta
 \left(H_{g-1}\bigl((\alpha,\beta)\bigr)
  +\sum_{g_1+g_2=g}H_{g_1}\bigl((\alpha)\bigr)H_{g_2}\bigl((\beta)\bigr)
 \right).
\end{align}
Then for $n=1$ and $g\ge1$ we immediately have $h_{g,(1)}=0$.
\end{proof}

\begin{proof}[{Proof of Lemma \ref{lem:hgn:known} (3)}]
For $n=2$ the equation \eqref{eq:cj:l=1} yields
$$
 (2g+1)H_g\bigl((2)\bigr)
=\dfrac{1}{2} 
 \left(H_{g-1}\bigl((1,1)\bigr)
  +\sum_{g_1+g_2=g}H_{g_1}\bigl((1)\bigr)H_{g_2}\bigl((1)\bigr)
 \right),
$$
which gives in terms of $h_{g,\vect{k}}$ 
\begin{align}\label{eq:lem:1/2:1}
 h_{g,(2)} = h_{g-1,(1,1)}+\delta_{g,0}/2.
\end{align}
For $\vect{k}=(1,1)$ the cut-and join equation gives 
$$
(2g+2)H_g\bigl((1,1)\bigr) = 2H_g\bigl((2)\bigr),
$$
in other words,
\begin{align}\label{eq:lem:1/2:2}
 h_{g,(1,1)} = h_{g,(2)}.
\end{align}
The relations \eqref{eq:lem:1/2:1} and \eqref{eq:lem:1/2:2} 
gives the desired consequence.
\end{proof}

\begin{proof}[{Proof of Lemma \ref{lem:hgn:known} (4)}]
The cut-and join equation for $\vect{k}=(1^n)$ gives 
$$
 (2g-2+2n) H_{g}\bigl((1^n)\bigr)
=\binom{n}{2} 2 H_{g}\bigl((2,1^{n-2})\bigr).
$$
In terms of $h_{g,\mu}$ it yields the desired result.
\end{proof}

We close this subsection with mentioning the following statement,
although it is not necessary for the proof of the main theorem.

\begin{lem}
For a length two partition $\mu=(\mu_1,\mu_2)$, we have
$$
 h_{0,(\mu_1,\mu_2)}= 
 \dfrac{1}{\sigma_{\mu_1,\mu_2}} \dfrac{(\mu_1+\mu_2)!}{\mu_1+\mu_2} 
 \dfrac{\mu_1^{\mu_1}}{\mu_1!} \dfrac{\mu_2^{\mu_2}}{\mu_2!}
$$
with
\begin{align}
 \sigma_{\mu_1,\mu_2} := 
 \begin{cases} 1 & \text{ if } \mu_1 \neq \mu_2 \\ 
               2 & \text{ if } \mu_1 = \mu_2 
 \end{cases}.
\end{align}
\end{lem}

\begin{rmk}
The factor $1/\sigma_{\mu_1,\mu_2}$ is missed 
in \cite[the line before (2.1)]{EMS:2011}.
\end{rmk}

\begin{proof}
The cut-and-join formula in this case becomes 
\begin{align*}
h_{0,(\mu_1,\mu_2)} 
=&\dfrac{1}{\sigma_{\mu_1,\mu_2}}n h_{0,(n)}
\\
 &+\dfrac{1}{2}\sum_{\alpha+\beta=\mu_1} \alpha \beta
  \biggl(
   \dfrac{\sigma_{\mu_2,\beta}}{\sigma_{\mu_1,\mu_2}}
    \binom{n-1}{\alpha-1} h_{0,(\alpha)} h_{0,(\mu_2,\beta)} 
  +\dfrac{\sigma_{\mu_2,\alpha}}{\sigma_{\mu_1,\mu_2}}
    \binom{n-1}{\beta-1} h_{0,(\mu_2,\alpha)} h_{0,(\beta)} 
   \biggr)
\\
 &+\dfrac{1}{2}\sum_{\alpha+\beta=\mu_2} \alpha \beta
  \biggl(
   \dfrac{\sigma_{\mu_1,\beta}}{\sigma_{\mu_1,\mu_2}}
    \binom{n-1}{\alpha-1} h_{0,(\alpha)} h_{0,(\mu_1,\beta)} 
  +\dfrac{\sigma_{\mu_1,\alpha}}{\sigma_{\mu_1,\mu_2}}
    \binom{n-1}{\beta-1} h_{0,(\mu_1,\alpha)} h_{0,(\beta)} 
   \biggr)
\end{align*}
with $n:=|\mu|$.
In the right hand, 
we know that $h_{0,(n)}=n^{n-3}$ and 
we may assume the result side by induction.
Then the equation has the following form:
\begin{align*}
h_{0,(\mu_1,\mu_2)} 
=\dfrac{1}{\sigma_{\mu_1,\mu_2}}
 \biggl(
  n^{n-2}
 + (n-1) \sum_{i=1}^{2}
   \mu_j^{\mu_j} \binom{n-2}{\mu_j-2}
   \sum_{\alpha=1}^{\mu_i-1} \dfrac{1}{n-\alpha}
   \binom{\mu_i-2}{\alpha-1} (\mu_i-\alpha)^{\mu_i-\alpha}\alpha^{\alpha-2}
 \biggr)
\end{align*}
with $j := 3-i$.
The summation can be calculated using a similar formula as \eqref{eq:fmp} and \eqref{eq:fmp:2}.
We omit the detail.
\end{proof}

\section{Integrality}
\label{sect:main}

In this section we prove 

\begin{thm}\label{thm:main}
The Hurwitz numbers $h_{g,\mu}$ are positive integers except for the case 
$h_{g,(1)} = 0$ for $g\ge1$ and $h_{g,(2)}=h_{g,(1,1)}=1/2$ for $g \ge 0$.
\end{thm}

First we deduce the following claim from the cut-and-join equation .

\begin{prop}\label{prop:polynomial}
For any $g \in \bbZ_{\ge0}$ and $\mu\in\calP$
we have
$$
 h_{g,\mu} \in
  \bbZ_{\ge0}[h_{0,(n)} \mid n \in \bbZ_{+}].
$$
\end{prop}

\begin{proof}
In terms of $h_{g,\vect{k}}$ the cut-and-join equation reads
\begin{equation}\label{eq:cj:h}
\begin{split}
 h_{g,\vect{k}}
=&\sum_{i < j} 
  \dfrac{\#\Aut\bigl(\vect{k}(\wh{i},\wh{j};k_i+k_j)\bigr)}{\#\Aut\bigl(\vect{k}\bigr)}
   (k_i+k_j) h_{g,\vect{k}(\wh{i},\wh{j};k_i+k_j)} 
\\
 &+\dfrac{1}{2}\sum_{i=1}^{\ell} 
   \sum_{\alpha+\beta=k_i} \alpha \beta
  \biggl(
   \dfrac{\#\Aut\bigl(\vect{k}(\wh{i};\alpha,\beta)\bigr)}{\#\Aut\bigl(\vect{k}\bigr)}
   h_{g-1,\vect{k}(\wh{i};\alpha,\beta)}
\\
 & \hskip 8em
   +\sum_{g_1+g_2=g}
    \sum_{\vect{l} \sqcup \vect{n} = \vect{k}(\wh{i})}
   \dfrac{\#\Aut\bigl(\vect{l}(\alpha)\bigr) \#\Aut\bigl(\vect{n}(\beta)\bigr)}
         {\#\Aut\bigl(\vect{k}\bigr)}
   \dfrac{(r-1)!}{r_1!r_2!}
    h_{g_1,\vect{l}(\alpha)} h_{g_2,\vect{n}(\beta)} 
   \biggr).
\end{split}
\end{equation}
Here we used the symbols 
$r:=r(g,\vect{k})$, $r_1:=r(g,\vect{l}(\alpha))$ and  
$r_2:=r(g,\vect{n}(\beta))$.
Note that $r-1 = r_1+r_2$.

As we saw at Remark \ref{rmk:cj:rec},
this equation reads 
$h_{g,\mu} \in \bbQ_{\ge0}[h_{0,(n)} \mid n \in \bbZ_{+}]$
since it decreases $r(g,\vect{k})$.
We will now check that the rational coefficients appearing in the right hand side 
actually sum up to integers.
In the following we use $m_i := m_i(\vect{k})$.

At the first term, if $k_i \neq k_j$, then we have
\begin{align*}
\text{coefficient of } \ 
h_{g,\vect{k}(\wh{i},\wh{j};k_i+k_j)}
&=m_{k_i} m_{k_j} 
 \dfrac{\#\Aut\bigl(\vect{k}(\wh{i},\wh{j};k_i+k_j)\bigr)}{\#\Aut\bigl(\vect{k}\bigr)}
 (k_i+k_j) 
=m_{k_i} m_{k_j} \dfrac{m_{k_i+k_j}+1}{m_{k_i}m_{k_j}} (k_i+k_j)
\in \bbZ_{+}. 
\end{align*}
If $k_i = k_j$, then
\begin{align*}
\text{coefficient of } \ 
h_{g,\vect{k}(\wh{i},\wh{j};2k_i)}
&=\binom{m_{k_i}}{2} 
 \dfrac{\#\Aut\bigl(\vect{k}(\wh{i},\wh{j};2k_i)\bigr)}{\#\Aut\bigl(\vect{k}\bigr)}
 2 k_i 
=\binom{m_{k_i}}{2} \dfrac{m_{2k_i}+1}{m_{k_i}(m_{k_i}-1)} 2 k_i
  \in \bbZ_{+}. 
\end{align*}

At the second term, if $\alpha \neq \beta$, then 
\begin{align*}
\text{coefficient of } \ 
h_{g,\vect{k}(\wh{i};\alpha,\beta)}
&=2m_{k_i} \dfrac{ \alpha \beta}{2} 
 \dfrac{\#\Aut\bigl(\vect{k}(\wh{i};\alpha,\beta)\bigr)}{\#\Aut\bigl(\vect{k}\bigr)} 
=m_{k_i} \alpha \beta \dfrac{(m_\alpha+1)(m_\beta+1)}{m_{k_i}}
  \in \bbZ_{+}. 
\end{align*}
Here the first factor $2m_k$ comes from the symmetry 
$(\alpha,\beta) \mapsto (\beta,\alpha)$ and the choice of $k_i$.
Similarly, if $\alpha=\beta$, then
\begin{align*}
\text{coefficient of } \ 
h_{g,\vect{k}(\wh{i};\alpha,\alpha)}
&=m_{k_i} \dfrac{\alpha^2}{2}
 \dfrac{\#\Aut\bigl(\vect{k}(\wh{i};\alpha,\alpha)\bigr)}{\#\Aut\bigl(\vect{k}\bigr)} 
=m_{k_i} \dfrac{\alpha^2}{2} \dfrac{(m_\alpha+1)(m_\alpha+2)}{m_{k_i}}
  \in \bbZ_{+}.
\end{align*}

Finally at the third term, 
\begin{align*}
&\text{coefficient of } \ 
h_{g_1,\vect{l}(\alpha)}h_{g_2,\vect{n}(\beta)} \\
&= \ve m_{k_i} 
    \binom{m_{k_i}-1}{m_{k_i}\bigl(\vect{l}\bigr)}
   \prod_{j \neq k_i } 
    \binom{m_j}{m_j\bigl(\vect{l}\bigr)}
   \cdot
   \dfrac{ \alpha \beta}{2} \binom{r-1}{r_1}
   \dfrac{\#\Aut\bigl(\vect{l}(\alpha)\bigr) \#\Aut\bigl(\vect{n}(\beta)\bigr)}
         {\#\Aut\bigl(\vect{k}\bigr)} \\
&= \ve m_{k_i} 
    \binom{m_{k_i}-1}{m_{k_i}\bigl(\vect{l}\bigr)}
   \prod_{j \neq k_i} 
    \binom{m_j}{m_j\bigl(\vect{l}\bigr)}
   \cdot
   \dfrac{\alpha \beta}{2} \binom{r-1}{r_1}
   \prod_j
    \dfrac{m_j\bigl(\vect{l}(\alpha)\bigr)! m_j\bigl(\vect{n}(\beta)\bigr)!}
          {m_j\bigl(\vect{k}\bigr)! }
\\
&= \ve (m_\alpha(\vect{l})+1)(m_\beta(\vect{n})+1) 
   \dfrac{ \alpha \beta}{2} \binom{r-1}{r_1} 
\end{align*}
with
\begin{align}\label{eq:ve} 
\ve :=  
  \begin{cases} 
   2 & \text{if } \ 
   \bigl(g_1, \vect{l}(\alpha)\bigr) \neq \bigl(g_2, \vect{n}(\beta)\bigr) \\
   1 & \text{otherwise}
  \end{cases}.
\end{align}
The last expression is a positive integer since 
if $\bigl(g_1, \vect{l}(\alpha)\bigr) = \bigl(g_2, \vect{n}(\beta)\bigr)$ 
then $r_1 = r_2$ so that $\binom{r-1}{r_1}=\binom{2r_1}{r_1} \in 2\bbZ_{+}$.

Therefore every coefficient in the right hand side of \eqref{eq:cj:h} is 
a non-negative integer,
and we have the conclusion.
\end{proof}

Now we turn to 

\begin{proof}[{Proof of Theorem \ref{thm:main}}]
We show the integrality by the induction on the number $r(g,\mu)$.
$r(g,\mu)=0$ requires $(g,(\mu))=\bigl(0,(1)\bigr)$, and $h_{0,(1)}=1$ 
by Lemma \ref{lem:hgn:known} (1).
$r(g,\mu)=1$ requires $(g,(\mu))=\bigl(0,(2)\bigr)$, and (this case is excluded, but) 
$h_{0,(2)}=1/2$ by Lemma \ref{lem:hgn:known} (3).

Assume that the statement holds for $r(g,\mu)<=r-1$,
and consider the case $r(g,\mu)=r$ with $r\le3$.
By (the proof of) the last Proposition \ref{prop:polynomial},
an obstruction of the integrality of $h_{g,\mu}$ 
occurrs when $h_{g',(2)}$ or $h_{g',(1,1)}$ with some $g'$ appears 
in the right hand side of the cut-and-join equation \eqref{eq:cj:h}.
We now switch the notation from $\mu$ to $\vect{k}$.

At the first term of the right hand side of \eqref{eq:cj:h},
$\vect{k}(\wh{i},\wh{j};k_i+k_j)=(2)$ holds 
if and only if $\vect{k}=(1,1)$, and we can ignore this case.
$\vect{k}(\wh{i},\wh{j};k_i+k_j)=(1,1)$ cannot hold,
so it is done.

At the second term,
$\vect{k}(\wh{i};\alpha,\beta)=(2)$ cannot hold.
$\vect{k}(\wh{i};\alpha,\beta)=(1,1)$ hols if and only if 
$\vect{k}=(2)$, and we can ignore this case.

The non-trivial consideration is necessary only at the third term.
Assume $\vect{l}(\alpha)=(2)$.
Then $\vect{l}=\emptyset$ and $\alpha=2$, 
so that $\beta=k_i-2$ and 
$\vect{n}(\beta)=\vect{k}(\wh{i},k_i-2)$.
Then, by the same calculation as in the proof of the previous 
Proposition \ref{prop:polynomial},
the term including 
$h_{g_1,\vect{l}(\alpha)} h_{g_2,\vect{n}(\beta)}$ is
\begin{align}\label{eq:thm:3-(2)}
&\ve  (m_\alpha(\vect{l})+1)(m_\beta(\vect{n})+1) 
 \dfrac{\alpha\beta}{2}\binom{r-1}{r_1}
 h_{g_1,\vect{l}(\alpha)} h_{g_2,\vect{n}(\beta)}
=\dfrac{\ve}{2} (m_\alpha(\vect{l})+1)(m_\beta(\vect{n})+1) 
 \beta \binom{r-1}{r_1} h_{g_2,\vect{n}(\beta)}
\end{align}
with $m_i := m_i(\vect{k})$ and $r_1:=r\bigl(g_1,\vect{l}(\alpha)\bigr)=2g_1+1$.
By the definition \eqref{eq:ve} of $\ve$ and by induction,
this expression is a non-negative integer if 
$\bigl(g_1, \vect{l}(\alpha)\bigr) \neq \bigl(g_2, \vect{n}(\beta)\bigr)$ 
and $h_{g_2, \vect{n}(\beta)} \in \bbZ_{\ge0}$.
Thus we have only to consider the following cases:
(i) 
$\bigl(g_1, \vect{l}(\alpha)\bigr) = \bigl(g_2, \vect{n}(\beta)\bigr)$ 
or 
(ii)$h_{g_2, \vect{n}(\beta)}=1/2$.
If (i) holds, then $\vect{n}(\beta)=(2)$, 
so that $\vect{k}=(4)$.
This case is already done in Proposition \ref{prop:int:hg(n)}.
If (ii) holds, then $\vect{n}(\beta)=(2)$ or $\vect{n}(\beta)=(1,1)$ 
by induction.
The first case is already excluded, so we may assume $\vect{n}(\beta)=(1,1)$.
Then $\vect{k}=(3,1)$, $\alpha=2$, $\beta=1$, 
and the expression \eqref{eq:thm:3-(2)} becomes
\begin{align*}
\eqref{eq:thm:3-(2)} = 
 \dfrac{2}{2} \cdot 1 \cdot 2 \cdot \binom{r-1}{r_1} \cdot \dfrac{1}{2},
\end{align*}
which is a positive integer.

The case $\vect{l}(\alpha)=(1,1)$ is similar.
Here we have $\vect{l}=(1)$, $\alpha=1$, $\beta=k_i-1$ 
and $\vect{n}(\beta)=\vect{k}(\wh{i};k_i-1)$.
The term including 
$h_{g_1,\vect{l}(\alpha)} h_{g_2,\vect{n}(\beta)}$ is
given by 
\begin{align}\label{eq:thm:3-(11)}
&\ve (m_\alpha(\vect{l})+1)(m_\beta(\vect{n})+1) 
 \dfrac{\alpha\beta}{2}\binom{r-1}{r_1}
 h_{g_1,\vect{l}(\alpha)} h_{g_2,\vect{n}(\beta)}
=\dfrac{\ve}{4} (m_\alpha(\vect{l})+1)(m_\beta(\vect{n})+1) 
 \beta \binom{r-1}{r_1} h_{g_2,\vect{n}(\beta)}
\end{align}
The cases we have to consider are 
(iii) 
$\bigl(g_1, \vect{l}(\alpha)\bigr) = \bigl(g_2, \vect{n}(\beta)\bigr)$ 
or 
(iv)$h_{g_2, \vect{n}(\beta)}=1/2$.
If (iii) holds, then $\vect{n}=(1)$, $\beta=1$ and $\vect{k}=(2,1,1)$.
Thus the expression \eqref{eq:thm:3-(2)} becomes
\begin{align*}
\eqref{eq:thm:3-(11)} = \dfrac{2}{4} \cdot 2 \cdot 2 \cdot 
 \binom{r-1}{r_1} \cdot \dfrac{1}{2},
\end{align*}
which is a positive integer.
The case (iv) is reduced to the case (ii) or (iii), and 
the proof is completed.
\end{proof}

During the study of the Hurwitz numbers, 
we found the following condition on the parity.

\begin{conj}\label{conj:parity}
For $|\mu|\ge3$ and $g\ge0$, we have 
\begin{align*}
\text{$h_{g,\mu}$ is odd} \Longrightarrow
\text{$r(g,\mu)$ is even, every part of $\mu$ is odd and $\ell(\mu) \le 2$}.
\end{align*}
\end{conj}

\begin{rmk}
We checked the condition for $r(g,\mu)\le 18$.
Unfortunately the converse does not hold.
The cases for which the converse does not hold and $r(g,\mu)\le 14$
are the following:
\begin{align*}
r(g,\mu)=8:\quad  
(g,\mu)=&\bigl(1,(7)\bigr), \bigl(1,(5,1)\bigr),\bigl(1,(3,3)\bigr),
\\
r(g,\mu)=10:\quad  
(g,\mu)=&\bigl(0,(7,3)\bigr),\bigl(1,(7,1)\bigr),
\bigl(1,(5,3)\bigr), \bigl(1,(9)\bigr),
\\
r(g,\mu)=12:\quad  
(g,\mu)=&\bigl(0,(7,5)\bigr),\bigl(1,(7,3)\bigr),
\bigl(3,(5,1)\bigr), \bigl(3,(3,3)\bigr),\bigl(3,(7)\bigr),
\\
r(g,\mu)=14:\quad  
(g,\mu)=&\bigl(0,(11,3)\bigr),\bigl(0,(7,7)\bigr),
\bigl(1,(7,5)\bigr)
\bigl(2,(9,1)\bigr), \bigl(2,(7,3)\bigr),\bigl(2,(5,5)\bigr),
\bigl(2,(11)\bigr),
\\
&\bigl(3,(7,1)\bigr), \bigl(3,(5,3)\bigr),\bigl(3,(9)\bigr).
\end{align*}
For these pairs, the Hurwitz number $h_{g,\mu}$ is even.
\end{rmk}


\end{document}